\documentclass[12pt]{amsart}
\usepackage{amsfonts, latexsym , hyperref, amsmath,amstext,amsthm,amssymb,amsxtra,
graphicx, subfigure}
\usepackage{verbatim}
\def\done{{1\hskip-2.5pt{\rm l}}}

\newtheorem{thm}{Theorem}

\newtheorem{claim}{Claim}[section]

\newtheorem{lem}{Lemma}[section]

\newtheorem{rmk}{Remark}[section]
\newtheorem{prop}{Proposition}[section]
\newtheorem{note}{Notation}
\newtheorem{obs}{Observation}

\newtheorem{defin}{Definition}[section]
\newtheorem{cor}{Corollary}[section]

\newcommand {\ds}{\displaystyle}
\newcommand{\norm}[1]{\left\Vert#1\right\Vert}

\newcommand{\R}{\mathbb R}
\newcommand{\C}{\mathbb C}
\newcommand{\N}{\mathbb N}
\newcommand{\E}{\mathbb E}
\newcommand{\Pro}{\mathbb P}
\def\ep{\epsilon}

\title[Zeroes of Translation-Invariant GAFs]{Zeroes of Gaussian Analytic Functions with Translation-Invariant Distribution}

\author{Naomi Feldheim}
\thanks{School of Mathematical Sciences,
Tel Aviv University, Tel Aviv, 69978, Israel.
Research supported by the Science Foundation of the Israel Academy
of Sciences and Humanities, grant 171/07. }

\begin{document}
 \maketitle

\begin{abstract}
We study zeroes of Gaussian analytic functions in a strip in the complex
plane, with translation-invariant distribution. We prove that the
horizontal limiting measure of the zeroes exists almost surely, and that it
is non-random if and only if the spectral measure is continuous (or
degenerate). In this case, the limiting measure is computed in terms of the
spectral measure. We compare the behavior with Gaussian analytic function
with symmetry around the real axis. These results extend a work by Norbert
Wiener.
\end{abstract}


\section{Introduction}
Following Wiener, we study zeroes of Gaussian analytic functions in a strip
in the complex plane, with translation-invariant distribution. Under certain
assumptions on the spectral measure, Wiener proved that the zeroes obey the
law of large numbers, and computed their horizontal density (limiting
measure). This result appears in his classical treatise with Paley
\cite[chapter X]{PW}. Wiener's proof is quite intricate; this may be why it
attracted little attention.

In this work, we simplify Wiener's arguments and remove unnecessary
assumptions on the spectral measure. We incorporate the result into a theorem
that guarantees the existence of the horizontal limiting measure in question, and
asserts it is not random if and only if the spectral measure is continuous or
consists of a single atom. Then we prove a counterpart of this theorem for a
natural class of Gaussian analytic functions which have a symmetry with
respect to the real axis.

For this purpose, we developed a general Edelman-Kostlan-type formula for
computing the average zero-counting measure of zeroes of a symmetric Gaussian
analytic function in some domain (see theorem \ref{thm-Rdensity} below). This
result extends those of Shepp and Vanderbei \cite{S&V}, Prosen \cite{Pros}
and Macdonald \cite{Mac}.

\emph{Acknowledgements:} This work was done during and following my master
thesis in Tel-Aviv university, under the attentive guidance of Mikhail Sodin.
I would like to thank him for his great devotion and fruitful influence on me
and the project. I also thank Boris Tsirelson, for teaching me most of the
probability I know in a most patient and inspiring way, and for contributing
the idea and proof in appendix \ref{app all segs}.

\subsection{ Gaussian analytic functions}\label{secsub GAF}
%By a Gaussian analytic function in a plane domain $D\subset \C$ we
%mean a random analytic function $f$ in $D$, such that for each
%finite set of points $z_1,\dots,z_n\in D$ the random variables
%$f(z_1),\dots,f(z_n)$ have a joint Gaussian distribution.

%This property is equivalent to the following definition, in which
%the random function is constructed as a linear combinations of
%analytic functions with independent, identically distributed (i.i.d)
%Gaussian coefficients \cite[Lemma 2.2.3]{GAF book}.

We will deal with two classes of random Gaussian analytic functions.

\begin{defin}\label{def1}
Let $D\subset \mathbb C$ be a domain, and let $\{\phi_n\}_{n\in\mathbb N}$ be
analytic functions in $D$ such that the series $\sum_n |\phi_n(z)|^2$
converges uniformly on compact subsets of $D$.

\begin{enumerate}
\item Let $a_n$ be independent standard complex Gaussian random variables
    ($a_n\sim \mathcal N_{\mathbb C}(0,1)$). The random series $\sum_n
    a_n \phi_n(z)$ is called a Gaussian Analytic Function (GAF, for
    short).

\item Let $b_n$ be independent standard \emph{real} Gaussian variables
    ($b_n\sim \mathcal N_{\mathbb R}(0,1)$). If the domain $D$ and the
    functions $\phi_n$ are symmetric w.r.t. the real axis (the latter
    means that $\phi(\overline z)=\overline{\phi(z)}$, $z\in D$) then the
    random series $\sum_n b_n \phi_n(z)$ is called a symmetric Gaussian
    Analytic Function.
\end{enumerate}
\end{defin}

Our assumptions on $\{\phi_n\}$ ensure that the sums above a.s. converge to an
analytic function in $D$ \cite[Chapter~2]{GAF book}. Throughout the paper we
assume that there is no $z_0\in D$ such that $\phi_n(z_0)=0$ for all $n\in\N$
(hence the function $f$ has no deterministic zeroes).

\smallskip
The covariance kernel of $f(z)$ is defined by
\begin{equation}\label{ker-base2}
 K(z,w)=\mathbb{E}(f(z)\overline{f(w)})=\sum_n \phi_n(z)\overline{\phi_n(w)}\,.
\end{equation}

The function $K(z,w)$ is positive definite, analytic in $z$, anti-analytic in
$w$, and obeys the law $K(z,w)=\overline{K(w,z)}$. It turns out that every
such function $K(z,w)$ of two variables $z,w\in D$ uniquely defines a GAF in
$D$.

If in addition $K(x,y)$ is real whenever $x,y\in D\cap\R$, then $K(z,w)$ also
uniquely defines a symmetric GAF with this kernel. We stress that a GAF and a
symmetric GAF with the same kernel are different random processes.

\subsection{Stationarity}\label{secsub stat}
We assume our domain is the $\Delta$-strip $D=D_\Delta
=\{|\text{Im}\:z|\!<\!\Delta\}$ with $0<\Delta\leq\infty$.
\begin{defin}
A GAF or symmetric GAF in a strip $D_\Delta$ is called {\em stationary}, if
it is distribution-invariant with respect
 to all horizontal shifts,
i.e., for any $t\in\mathbb R$, any $n\in\mathbb N$, and any $z_1,\dots,z_n\in
D$, the random $n$-tuples
\[
\bigl( f(z_1),\dots,f(z_n) \bigr) \qquad {\rm and} \qquad \bigl(
f(z_1+t),\dots,f(z_n+t) \bigr)
\]
have the same distribution.
\end{defin}

If $f(z)$ is stationary in the $\Delta$-strip, then for any $x,y\in\R$ the
covariance $\E (f(x)\overline{f(y)})$ depends on $(x-y)$ only, so that
$K(x,y) = r(x-y)$ for some real-analytic function $r:\R\to\C$. From this we
deduce $K(z,w)=r(z-\bar w )$ (both functions are analytic in $z$,
anti-analytic in $w$, and coincide for $z,w\in\R$), and so $r(t)$ has an
analytic continuation to the $2\Delta$-strip $D_{2\Delta}$.

Since $r(t)$ is continuous and positive-definite, it is the Fourier transform
of a positive measure $\rho$ (Bochner's Theorem):
$$r(t) = \int_\R e^{2\pi i t \lambda}d\rho(\lambda).$$
The measure $\rho$ is called the spectral measure of the process $f(z)$.

Since $r(t)$ has an analytical extension to the $2\Delta$-strip,
$\rho(\lambda)$ has a finite exponential moment \cite[Chapter 2]{LO}:

\begin{equation}\label{cond}
\text{for each }\Delta_1<\Delta,\:\: \int_{-\infty}^{\infty} e^{2\pi
\cdot 2\Delta_1 |\lambda|} d\rho(\lambda) < \infty\,.
\end{equation}

In fact, condition \eqref{cond} is also sufficient for $r(t)$ to have an
analytic extension to the $2\Delta$-strip. Therefore, beginning with a finite
positive measure $\rho$ obeying \eqref{cond}, we can construct a kernel by
\begin{equation}\label{eq K-rho}
K(z,w) = \int_\R e^{2\pi i (z-\bar w)\lambda} d\rho(\lambda).
\end{equation}
which defines in its turn a stationary GAF in the $\Delta$-strip.

What measures could be spectral measures of a symmetric GAF? As we mentioned
earlier, a kernel $K(z,w)$ defines a symmetric GAF if and only if it is real
for $z,w\in\R$; By relation \eqref{eq K-rho} this is equivalent to $\rho$
being symmetric with respect to the origin.

Finally, we mention that a random GAF or symmetric GAF may be explicitly
constructed, as follows, from its spectral measure $\rho$. If
$\{\psi_n(z)\}_n$ comprise an orthonormal basis in $L^2_\rho(\R)$, then their
Fourier transforms
\[\phi_n(z)=\widehat {\psi_n}(z) =
\ds \int_\R e^{2\pi i z \lambda }\psi_n(\lambda) d\rho(\lambda)
\]
comprise a basis in the Hilbert space $\mathcal F \{L^2_\rho(\R)\}$ (the
Fourier image of $L^2_\rho(\mathbb R)$ with the scalar product transferred
from $L^2_\rho(\mathbb R)$). Then we have
\[
r(z-\overline w)=\mathbb{E}(f(z)\overline{f(w)})=\sum_n
\phi_n(z)\overline{\phi_n(w)}\,.
\]
Therefore, when used in definition~\ref{def1}, the basis $\bigl\{ \phi_n
\bigr\}$ will give us a random function with the desired kernel.

\section{Results and Discussion}\label{sec res}

\subsection{Main Theorem} It will be convenient to introduce some notation:

\begin{note}
\emph{(zero-set, zero-counting measure)} {\rm  Let $D\subset\mathbb{C}$ be a
region, and $f$ a holomorphic function in $D$. Denote the zero-set of $f$
(counted with multiplicities) by $ Z_f $, and the zero-counting measure by
$n_f$, i.e.,
\[
\forall \phi \in C_0^\infty(D), \qquad \int_D \phi(z) dn_f(z) =
\sum_{z\in Z_f} \phi(z)\,.
\]
We use the abbreviation $\displaystyle n_f(B)=\int_B dn_f(z)$ for the number
of zeroes in a Borel subset $B\subset D$.}

\end{note}

\begin{note}\label{note funcs}
{\rm Let $y\in(-\Delta,\Delta)$. For a stationary GAF or
symmetric-GAF in $D_\Delta$ with kernel $K(z,w)$, denote
\[
\psi(y) =K(iy,iy) = \displaystyle\int_{-\infty}^\infty e^{-4\pi
y\lambda}d\rho(\lambda)\,.
\]
In the case of a GAF, define the function
\begin{equation} \label{eq L}
L(y) = \frac{d}{dy}
\left(\frac{\psi'(y)}{4\pi\psi(y)}\right)
=-\frac{d}{dy} \left(\frac{\int_{-\infty}^\infty \lambda
e^{-4\pi y\lambda}d\rho(\lambda)} {\int_{-\infty}^\infty e^{-4\pi
y\lambda}d\rho(\lambda)}\right).
\end{equation}
In the case of a symmetric-GAF, define for $y\neq 0$ the function
\begin{equation}\label{eq S}
S(y) = \frac{d}{dy} \left(\frac{\psi'(y)}{4\pi\sqrt{\psi(y)^2-\psi(0)^2}}\right)= -
\frac{d}{dy}
\left(\frac{\int_{-\infty}^\infty \lambda e^{-4\pi
y\lambda}d\rho(\lambda)} {\sqrt{\left(\int_{-\infty}^\infty e^{-4\pi
y\lambda}d\rho(\lambda)\right)^2-\left(\int_{-\infty}^\infty
d\rho(\lambda)\right)^2 }}\right)\,,
\end{equation}
and the positive number
\begin{equation}\label{eq R}
R = \frac{1}{4\pi}
\sqrt{\frac{\psi''(0)}{\psi(0)}}=2\sqrt{\displaystyle\frac
{\int_{-\infty}^\infty \lambda^2
d\rho(\lambda)}{\int_{-\infty}^\infty d\rho(\lambda)}}.
\end{equation}
}
\end{note}

Finally, a stationary GAF is \emph{degenerate} if its spectral measure
$\rho_f$ consists of exactly one atom. Similarly a stationary symmetric GAF
is degenerate if $\rho_f$ consists of two symmetric atoms (i.e., $\rho_f = c(
\delta_q + \delta_{-q})$ for some $c, q>0$).

The following theorem is our main result. Denote by $m_1$ the linear Lebesgue
measure.

\begin{thm}\label{main}
Let $f$ be a stationary non-degenerate GAF or symmetric GAF, in the strip
$D_\Delta$ with $0<\Delta\leq\infty$. Let $\nu_{f,T}$ be a non-negative
locally-finite random measure on $(-\Delta,\Delta)$ defined by

\begin{equation*}
    \displaystyle \nu_{f,T}(Y)=\frac{1}{T}\,n_f([0,T)\times Y), \:\:Y\subset (-\Delta,\Delta)
\end{equation*}

Then:

\smallskip\par\noindent{\rm (i)}
Almost surely, the measures $\nu_{f,T}$ converge weakly to a measure $\nu_f$
when $T\rightarrow \infty$.

\smallskip\par\noindent{\rm (ii)}
The measure $\nu_f$ is not random (i.e. $\operatorname{var} \nu_f =0 $) if
and only if the spectral measure $\rho_f$ has no atoms.

\smallskip\par\noindent{\rm (iii)}
If the measure $\nu_f$ is not random, then:\\
\begin{align*}
&\nu_f = L\, m_1, &\text{ if } f \text{ is a GAF,} \\
&\nu_f = S\, m_1 + R\,\delta_0, &\text{ if } f \text{ is a symmetric-GAF, }\\
\end{align*}
where $\delta_0$ is the unit point measure at the origin.
\end{thm}

\vspace{3mm} The measure $\nu_f$ is referred to as "the horizontal limiting
measure of the zeroes of $f$", or simply "the limiting measure". In the
discussion and examples that follow, we assume the normalization
$\psi(0)=\int_\R d\rho(\lambda)=1$.
\begin{rmk}
{\rm The limiting measure $\nu_f$ might have atoms. Generally speaking, the
weak convergence in the theorem guarantees that $\nu_{f,T}([a,b))$ converges
to $\nu_f([a,b))$ for all $a,b\in(\-\Delta, \Delta)$ with a possible
exception of an at most countable set, which corresponds to atoms of the
limiting measure $\nu_f$. Yet, due to stationarity, in our case the limit
exists \emph{on all intervals}. We prove, as an example, the following result:

\begin{prop}\label{prop all segs}
Almost surely, for any $a,b\in(-\Delta,\Delta)$, we have:
$$\ds
\lim_{T\to\infty}\nu_{f,T}([a,b))=\nu_f([a,b)).$$
\end{prop}

The proof is included in appendix \ref{app all segs}. }
\end{rmk}

\begin{rmk}
{\rm The part of the theorem pertaining to GAFs extends the aforementioned
Wiener's theorem. In his work, Wiener assumed that the spectral measure
$\rho$ has the $L^2$-density $d\rho(\lambda)=|\phi(\lambda)|^2d\lambda$, that
satisfies convergence conditions:}
\smallskip\par\noindent
For any $|y|<\Delta$,
\[
\int_{-\infty}^\infty (1+x^2)^2\, |\widehat{\phi}(x+iy)|^2
dx<\infty,
\]
and
\[
\int_{-\infty}^\infty (1+x^2)\, |(\widehat{\phi})'(x+iy)|^2
dx<\infty.
\]
{\rm As above, $\widehat{\phi}$ is the Fourier transform of $\phi$.}
Under these assumptions, Wiener proved that the limiting measure $\nu_f$ exists and equals $L m_1$ where L is defined by \eqref{eq L}.

\end{rmk}

\begin{rmk}\textbf{(atomic spectral measure)} \label{rmk atom}
{\rm

Consider a spectral measure consisting of two atoms:
\[
\rho = \frac{1}{2}(\delta_{-q} + \delta_q)\,.
\]
%Here $K(z,w) = \cos(q(z-\overline{w}))$.
The corresponding GAF is $f(z)=(\zeta_1 e^{-2\pi iqz} + \zeta_2 e^{2\pi
iqz})/\sqrt{2}$, where $\zeta_1 , \zeta_2 \sim \mathcal{N}_\mathbb{C}(0,1)$,
independently. The zeroes of such a function are
\[
z_k = \frac{1}{4\pi q} \left[ \arg(\frac{\zeta_2}{\zeta_1})+2\pi k -
i\log \left|\frac{\zeta_2}{\zeta_1} \right| \right]\,, \qquad
k\in\mathbb{N}\,.
\]
We see that all zeroes lie on the same (random) horizontal line, equally
spaced upon it. The height of this horizontal line is a non-degenerate random
variable, and so in this example $\nu_f$ is indeed random.

\medskip
For symmetric GAFs, the spectral measure above is degenerate (all zeroes of
the corresponding function shall be real). We mention that it is possible to construct a random analytic
function with continuous spectrum, for which a given asymptotic proportion of
zeroes lie on the real line. For this, choose a continuous symmetric spectral
measure, sufficiently close to the degenerated measure $\delta_q +
\delta_{-q}$.
%Using Theorem \ref{main} and proposition \ref{prop all segs}, we can show that the corresponding symmetric GAF has, say, $99\%$ out of its zeroes in a long rectangle lying on the real line.

}
\end{rmk}

\begin{rmk}\textbf{(behavior near the boundary and near the real line.)}
{\rm We observe that $S(y)$ and $L(y)$ have the same asymptotic behavior as
$y$ approaches the boundary $\pm \Delta$. Therefore, zeroes of a GAF and of a
symmetric GAF with the same kernel behave similarly near the boundary of the
domain of definition.

For a symmetric GAF, we observe a "contraction" of the zeroes to the real
line: there are zeroes on the line itself, but they are scarce as we approach
it from below or above (see figure \ref{fig-all} below). This is confirmed by
a straightforward computation, which shows that $S(y) = O(y)$, as
$y\rightarrow 0$.

\begin{comment} %computation
In light of Theorem \ref{main}, this will be clear if we prove that
$\ds\lim_{y\to 0} S'(y) = 0$ (as the derivative indeed exists).

Differentiating \eqref{eq S}, we get:
$$S^\prime(y) = \frac{-\psi(y) \, \psi'(y)^2 + (\psi(y)^2-1)\,\psi''(y) }{\left(\psi(y)^2-1\right)^{3/2}} $$
$\psi(y)$ is real-analytic, and has a Taylor expansion of the form: $\psi(y)
= 1 + \frac{\alpha}{2}y^2 + O(y^4)$, as $y\rightarrow 0$ (the odd terms
disappear because $\rho$ is symmetric; $\alpha$ is a constant depending on
$\rho$). Therefore,
\begin{align*}
S^\prime(y) = \frac{-\left(1+\frac{\alpha}{2} y^2 +
O(y^4)\right)\left(\alpha^2y^2+O(y^4)\right)+ \left(\alpha y^2
+O(y^4)\right)\left(\alpha+O(y^2)\right)}{(\alpha y^2)^{3/2} +
O(y^4)} \\= \frac{-\alpha^2y^2 + O(y^4)
+\alpha^2y^2+O(y^4)}{O(y^3)}=O(y),  \:\:\: y\rightarrow 0
\end{align*}
which proves our claim.
\end{comment}
}
\end{rmk}

\medskip

\subsection{Expected Zero-Counting Measures}
In part(iii) of the theorem, the limit $\nu_f(a,b)$ is actually the
expectation $\E n_f([0,1]\times[a,b])$. In order to calculate this quantity
in the GAF case, we use the following classical formula, which appeared in
Edelman and Kostlan's joint work on random polynomials \cite{EK}. Several
proofs of this formula are known (\cite{GAF book}, chapter 2).
\begin{thm}\emph{(Edelman-Kostlan formula)}
 For a Gaussian Analytic Function $f$ with covariance kernel $K(z,w)$, the expected zero-counting measure is given by
\begin{equation}\label{eq EK}
\mathbb{E}(n_f) = \frac{1}{4\pi}\bigtriangleup \log K(z,z).
\end{equation}
\end{thm}

This should be understood as equality of measures in the following sense: for
any compactly supported $h\in C^\infty(D)$,
$$\mathbb{E}\int_D h(z) dn_f(z) = \frac{1}{4\pi}\int h(z)
\:\bigtriangleup \log K(z,z) dm_2(z).$$
Here and throughout this paper, $m_2$ denotes the planar Lebesgue measure.
\vspace{3mm}

The proofs of this formula depend inherently on the fact that $f(z)$ is a
{\em complex} Gaussian random variable
 for all $z$, which fails for the symmetric GAF.
To that end, we prove the following result, that extends previous results by
Shepp and Vanderbei \cite{S&V}, Prosen \cite{Pros} and Macdonald \cite{Mac}.

\begin{thm}\label{thm-Rdensity}
For a symmetric GAF $f$ on some region with covariance kernel $K(z,w)$, the
expected zero-counting measure is given by
\begin{equation}\label{eq-Rdensity}
\mathbb{E}(n_f) =  \frac{1}{4\pi} \bigtriangleup\log \left(
K(z,z)+\sqrt{K(z,z)^2 - |K(z,\overline{z})|^2} \right),
\end{equation}
where the Laplacian is taken in the distribution sense.
\end{thm}
Notice that stationarity is not assumed in the last two theorems. Moreover,
this formula combines information about real and complex zeroes.

\section{Examples}\label{sec ex}

\subsection{Paley-Wiener Process (Sinc-kernel Process)}\label{ex-sin}
Consider the spectrum
\[
d\rho_a(\lambda) =\frac{1}{2a} \chi_{[-a ,a]}(\lambda)d\lambda\,,
\qquad a>0.
\]
Condition \eqref{cond} holds for any $\Delta>0$, so the sample function $f$
is entire. The kernel is:
$$K(z,w)=\frac{\sin(2\pi a (z-\overline w))}{2\pi a(z-\overline w)}=r(z-\overline{w})$$
%Since $\rho$ is symmetric, both a GAF and a symmetric GAF exist with this kernel.
%A base for construction of either of them is

A base for construction of the GAF (in the sense of definition \ref{def1}) is
$$\phi_n(z)=\frac{\sin(2\pi a z)} {2\pi a z - n\pi}, \qquad n\in\mathbb{Z}.$$

%% Indeed, at n/2 we have zeroes of the top and bottom of the expression.
%%It can be built from the same basis $\phi_n(z)$ we have found (since it is real on $\mathbb{R}$).

This example yields a surprising construction of a random series of simple
fractions with known poles and stationary zeroes: Take for instance $a=1$.
Our function is
\[
f(z) = \sum a_n\, \frac{\sin(2\pi z)} {2\pi z - n\pi}\,
\]
where $\{a_n\}$ are independent Gaussian random variables. Almost surely,
$Z_f \cap \frac{1}{2}\mathbb{Z}= \emptyset$, so we may divide by $\sin (2\pi
z)/\pi$ and get the random series
\[
g(z)=\sum \frac {a_n}{2 z - n}\,.
\]
The poles of $g$ are known (and lie on a one-dimensional lattice), but its
zeroes are a random set invariant to \emph{all} horizontal shifts!

\vspace{2.5mm} Using Theorem \ref{main} for $\rho_a$, we get that the
zero-counting measure has the following density of zeroes:
$$ L_a\left(\frac{y}{4\pi a}\right) =4\pi a^2\,\frac{d}{dy} \left(\coth y-\frac{1}{y}\right).$$
%$$ L_a\left(\frac{y}{4\pi a}\right) = a \left(\coth y-\frac{1}{y}\right).$$

%The density itself can be computed here to be
%$$\mathbb E dn_f\left(\frac{x+iy}{4\pi a}\right) = a \left(\frac{1}{y^2}-\frac{1}{\sinh^2 y}\right) dm_{\mathbb R ^2}(x,y).$$

Similarly, the symmetric GAF with the same spectral measure has the
continuous density of zeroes
$$ S_a\left(\frac{y}{4\pi a}\right) = 4\pi a^2 \,\frac {d}{dy}\left(\frac{\cosh y -\frac{\sinh y}{y}}{\sqrt{\sinh^2 y-y^2}}\right)$$
%% a \,\frac{\coth y-\frac{1}{y}}{\sqrt{1-y^2/\sinh^2 y}}$$
plus an atom at $y=0$, of size $$R = \frac{a}{\sqrt 3}.$$

Figure \ref{fig-w} represents the graphs of the continuous
densities for the parameter $a=\frac{1}{4\pi}$.

\begin{figure}[htp]\label{fig-all}
  \begin{center}
    \subfigure[Paley-Wiener]{\label{fig-w}\includegraphics[scale=0.5]{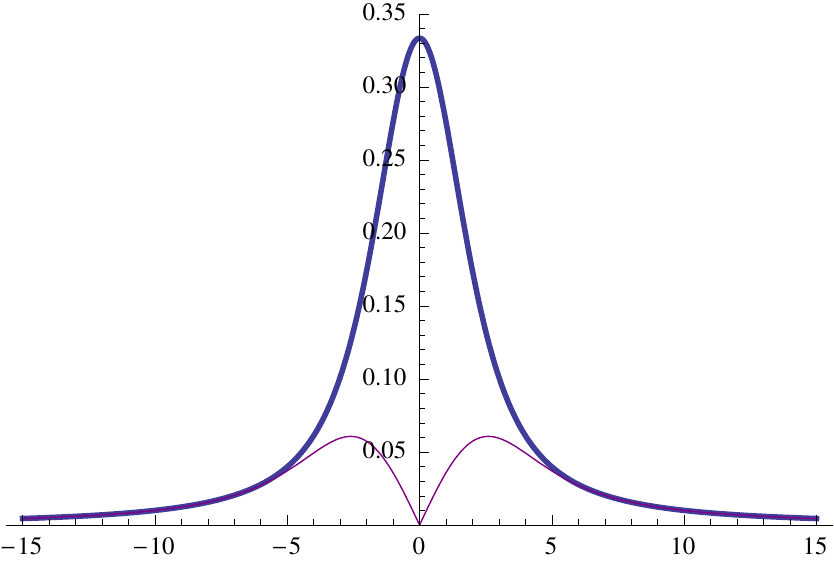}}
    \subfigure[Fock-Bargmann]{\label{fig-FB}\includegraphics[scale=0.5]{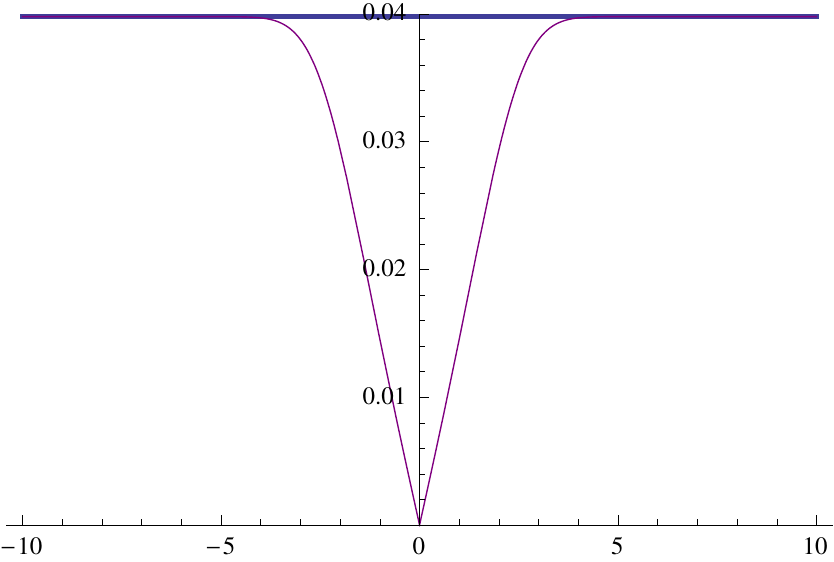}}
    \subfigure[exponential spectrum]{\label{fig-exp}\includegraphics[scale=0.5]{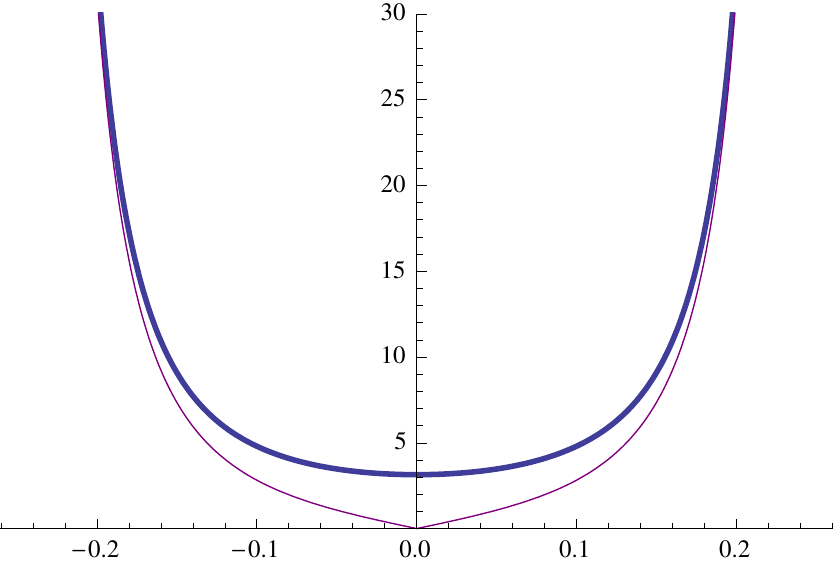}}
  \end{center}
  \caption{Horizontal density of zeroes for GAF and symmetric GAF models with the same kernel.
  In each model, the lower graph represents the continuous component of the mean zero counting-measure for the symmetric GAF
  (the atomic part is an atom at $y=0$, which is not graphed).
  The upper graph represents the continuous (and only) part of this measure for the appropriate GAF.}
  \label{fig:edge}
\end{figure}

\subsection{Fock-Bargmann Space (Gaussian Spectrum)}

Set
\[
d\rho_a(\lambda) = \frac 1 {a\sqrt{\pi}}e^{-\lambda^2/a^2}
d\lambda\,, \qquad a>0.
\]
Once again, $f$ is entire. The Fourier transform of the measure is
\[
r(z)=\frac 1 {a\sqrt{\pi}} \int_{-\infty}^{\infty}
e^{-\lambda^2/a^2}e^{2\pi i\lambda z}d\lambda = e^{-a^2 \pi^2 z^2},
\]
therefore the covariance kernel is:
$$K(z,w)= e^{-a^2\pi^2 (z-\overline w)^2}.$$
This space has an orthonormal basis of the form
$\frac{(bz)^n}{\sqrt{n!}}e^{-cz^2} $, where $b =\sqrt 2\frac a \pi$ and $c =
-\frac {a^2}{\pi^2}$.

% Here Psi(y) = Exp[(2a \pi y)^2]

In this model, the density of zeroes is constant:
$$ L_a\left(\frac y {2\pi a}\right)= 2\pi a^2$$
%when considering complex coefficients:
%$$\mathbb E dn_f\left(\frac{x+iy}{4\pi a}\right) =  \frac{a}{2}\ dm_{\mathbb{R}^2}(x,y).$$

This is the only model with Lebesgue measure as expected counting measure of
zeroes. For more information about this model and why the distribution of
zeroes determines the GAF, see \cite{ST1}, \cite{S&N-intro},
\cite{S&N-survey} or \cite[Chapters 2.3, 2.5]{GAF book}.

However, for the real coefficients case the continuous part of the limiting
measure has density
$$ S_a \left(\frac y {2\pi a}\right) = 2\pi a^2 \,\frac d{dy}\left( \frac{e^{y^2}}{\sqrt{e^{2y^2}-1}} \right)$$
and the atom at $y=0$ is of size $\sqrt 2 a$.

Both continuous densities are graphed in \ref{fig-FB}, for the parameter $a=\frac{1}{4\pi}$.

\subsection{Exponential Spectrum}
Consider a symmetric measure with exponential decay:
\[
d\rho(\lambda) = \text{sech}(\pi\lambda) d\lambda =
\frac{1}{\cosh(\pi \lambda)} d\lambda\,,
\]
then $r(z) = \text{sech}(\pi z)$ as well. This model is valid in the strip
$-\frac{1}{4}<\text{Im}(z)<\frac{1}{4}$. Here
$$L(y) =\frac {\pi}{\cos^2(2\pi y)}.$$

For the symmetric GAF in this model, we have
%before derivative $$\frac{1}{2}\frac{\text{sign}(y)}{\cos(2\pi y)}.$$
$$S(y)=\frac{\pi|\sin(2\pi y)|}{\cos^2(2\pi y)}$$
%so the density is
%$$\mathbb E_{\mathbb R} dn_f(x+i y) = \pi\, \text{sign}(y) \frac{\sin(2\pi y)}{\cos^2(2\pi y)}dm_{\mathbb{R}^2}(x,y) + dm_{\mathbb R}.$$
We see that the zeroes concentrate near the boundaries of the region of
convergence (figure \ref{fig-exp}).

\section{Proof of Theorem \ref{thm-Rdensity} - Zero-Counting Measure for a symmetric GAF}\label{sec RGAP-density}
In this section we prove Theorem \ref{thm-Rdensity}. Similar formulas were
proved in specific cases. Our proof follows Macdonald \cite{Mac}, who has
considered random polynomials (also in multi-dimensional case). A novelty is
in extension of his result to arbitrary symmetric GAFs.

Recall that for any analytic function $f$ (not necessarily random) in a
domain $D$ we have
$$ n_f = \frac{1}{2\pi}\bigtriangleup \log |f|.$$
This is understood in the distribution sense.

Using this for our random $f$, we would like to take expectation of both
sides, to get:
\begin{equation}\begin{aligned}\label{Justify}
\mathbb{E} \left[ \int_X h(z) dn_f(z)\right] = \:
& \mathbb{E} \left[ \frac{1}{2\pi}\int_X \bigtriangleup h(z) \log |f(z)| dm_2(z) \right] =\\
& \frac{1}{2\pi}\int_X \bigtriangleup h(z) \mathbb{E}\left[\log
|f(z)|\right] dm_2(z),
\end{aligned}
\end{equation}
where $m$ denotes the Lebesgue measure in $\C$. The last equality is
justified by Fubini's Theorem, as we will show in the end of this section.
Thus we can conclude that (in the weak sense):
\begin{equation} \label{Enf}
\mathbb{E}(n_f) = \frac{1}{2\pi}\bigtriangleup
\mathbb{E}\log|f|.
\end{equation}

Let us return to our setup: $f$ is a random function generated by a basis
$\phi_k(z)$ of holomorphic functions, each real on $\mathbb{R}$, and such
that the sum $\sum_k |\phi_k(z)|^2$
converges locally-uniformly. %[TODO: Reference to chapter]
Denote $\phi_k(z) = u_k(z)+i v_k(z)$ where $u_k, v_k$ are real functions. Our
random function is decomposed thus:
$$f(z)=\sum b_k \phi_k(z) = \sum b_k u_k(z) + i \sum b_k v_k(z)=u(z)+i v(z),$$
where $b_k \sim \mathcal{N}_{\mathbb{R}} (0,1)$ are \emph{real} Gaussian standard variables. $(u(z),v(z))$ have a joint Gaussian distribution, with mean (0,0) and covariance matrix
$$\Sigma =
 \left( \begin{array}{c c}
\sum u_k^2  &  \sum u_k v_k \\
\sum u_k v_k & \sum v_k^2
\end{array}\right). $$

\begin{lem}\label{eigenv}
 The matrix $\Sigma$ from above has 2 eigenvalues:
$$\lambda_{2,1} = \frac{K(z,z)\pm |K(z,\overline{z})|}{2}$$
where $K(z,w) = \sum\phi_k(z) \overline{\phi_k(w)}= \sum \phi_k(z)\phi_
k(\overline{w})$.
 \end{lem}

\begin{proof}
For any complex number $\phi = u+i v$, we have:
%$$ |\phi|^2 = u^2+v^2 , \:\:\: \phi^2 = u^2-v^2 - 2iuv $$
%Therefore,
$$ u^2 = \frac{1}{2}\left(|\phi|^2 + \text{Re}(\phi^2) \right),
\;\;
   v^2 = \frac{1}{2}\left(|\phi|^2 - \text{Re}(\phi^2)\right),
   \;\;
   uv = \frac{1}{2} \text{Im}(\phi^2).$$

Applying this, we can rewrite $\Sigma$ as
$$\Sigma =
 \left( \begin{array}{c c}
\frac{1}{2}\left(\sum|\phi_k|^2 + \text{Re}\sum \phi_k^2 \right)&
\frac{1}{2}\text{Im}\sum\phi_k^2 \\
\frac{1}{2}\text{Im}\sum\phi_k^2 & \frac{1}{2}\left(\sum|\phi_k|^2 -
\text{Re}\sum \phi_k^2 \right)
\end{array}\right), $$
and then calculate it's determinant and trace:
\begin{equation}\label{det trace}
\begin{aligned}
\lambda_1\lambda_2 = \det\Sigma &=
\frac{1}{4}\left((\sum|\phi_k|^2)^2-(\text{Re}\sum(\phi_k^2))^2 -
(\text{Im}\sum(\phi_k^2))^2 \right)\\
 &=\frac{1}{4}\left(K(z,z)^2-|K(z,\overline{z})| \right)\\
\lambda_1+\lambda_2 = \text{trace }\Sigma &=\sum|\phi_k|^2 = K(z,z)
\end{aligned}\end{equation}
The lemma follows.
\end{proof}

Using the law of bi-normal distribution, we get:
\begin{align}
\mathbb{E}[\log |f(z)|] = \frac{1}{2\pi \sqrt{\det\Sigma}}
\iint_{\mathbb{R}^2} \log(\sqrt{x^2+y^2})
e^{-\frac{1}{2} (x,y)\Sigma^{-1} (x,y)^{T}} dx\, dy = \label{eq-myint}\\
\frac{1}{2\pi \sqrt{\det\Sigma}} \iint_{\mathbb{R}^2}
\log(\sqrt{x^2+y^2}) e^{-\frac{1}{2} (\lambda_1^{-1} x^2
+\lambda_2^{-1} y^2)} dx\, dy \notag
\end{align}

%The last step is achieved by the change of variables
%$\left(\begin{array}{c}x'\\y'\end{array}\right)=U
%\left(\begin{array}{c}x\\y \end{array}\right)$, where $U$ is the
%orthogonal matrix diagonalizing $\Sigma^{-1}$.

Applying to the last integral the change of variables $x=u\sqrt{\lambda_1},\;
y =w\sqrt{\lambda_2}$, with Jacobian $\sqrt{\lambda_1
\lambda_2}=\sqrt{\det\Sigma}$, we have
\begin{align*}
\mathbb{E}[\log |f(z)|] = \frac{1}{2\pi}\iint_{\mathbb{R}^2}
\log(\sqrt{\lambda_1 u^2 +\lambda_2 w^2}) e^{-(u^2+w^2)/2}du\,dw
\end{align*}

Now, changing to polar coordinates $u = r\cos \theta, \: w=r\sin\theta$ we
get:
\begin{align*}
\mathbb{E}[\log |f(z)|]
%&=
%\frac{1}{2\pi}\int_0^{2\pi}\int_0^{\infty} \log(\sqrt{\lambda_1 r^2\cos^2\theta +\lambda_2 r^2\sin^2\theta})e^{-r^2/2}\, r\,dr\,d\theta = \\
%&=\frac{1}{2\pi}\int_0^{2\pi}\int_0^{\infty}
%\left(\log(\sqrt{\lambda_1 \cos^2\theta +\lambda_2 \sin^2\theta})
%+ \log r\right) e^{-r^2/2}\, r \,dr\,d\theta = \\
%&=\frac{1}{2\pi}\int_0^{2\pi} \log(\sqrt{\lambda_1 \cos^2\theta
%+\lambda_2 \sin^2\theta})\,d\theta \cdot \int_0^{\infty}
%e^{-r^2/2}r\, dr + C=\\
=\frac{1}{2\pi}\int_0^{2\pi} \log(\sqrt{\lambda_1 \cos^2\theta
+\lambda_2 \sin^2\theta})\,d\theta + C,
\end{align*}
where $C$ is a constant which does not depend on the point $z$ (i.e., is
independent of $\lambda_1$ and $\lambda_2$). In the following, we write $C$
for any such constant (which may be different each time we use this symbol).
These constants will vanish when we apply Laplacian (recall \eqref{Enf}).
%%From now on, we shall discard such constants, and write $\approx$
%%instead of $=$ when we do so.

So, the integral we should compute is:
\begin{align*}
&\int_0^{2\pi}  \log\left|\sqrt{\lambda_1} \cos\theta
+i \sqrt{\lambda_2} \sin\theta \right|\,\frac{d\theta}{2\pi} =\\
%&\int_0^{2\pi} \log\left|(\sqrt{\lambda_1}+\sqrt{\lambda_2})e^{i\theta}+
%(\sqrt{\lambda_1}-\sqrt{\lambda_2})e^{-i\theta}\right|\,\frac{d\theta}{2\pi} +C=\\
& \log(\sqrt{\lambda_1}+\sqrt{\lambda_2}) + \int_0^{2\pi} \log
\left\vert
e^{2i\theta}+\frac{\sqrt{\lambda_1}-\sqrt{\lambda_2}}{\sqrt{\lambda_1}+\sqrt{\lambda_2}}
\right\vert \,\frac{d\theta}{2\pi}+C.
\end{align*}

The remaining integral is computed easily by Jensen's formula for the
function $z^2+c$, where
$c=\frac{\sqrt{\lambda_1}-\sqrt{\lambda_2}}{\sqrt{\lambda_1}+\sqrt{\lambda_2}}<1$.
Indeed, it has two zeroes in the unit circle, denoted $a_1$ and $a_2$, and
so:
$$\int_0^{2\pi} \log \left\vert e^{2i\theta}+c\right\vert \,\frac{d\theta}{2\pi} =
\log|\phi(0)| - \log |a_1| -\log|a_2| = \log|c|-2\log\sqrt{|c|} =
0.$$

Recalling \eqref{Enf} and using the relations \eqref{det trace}, we arrive at
 \begin{equation*}
\mathbb{E}(dn_f) = \frac{\bigtriangleup}{2\pi}\,
\frac{1}{2}\log\left(\lambda_1 + \lambda_2 + \sqrt{4
\lambda_1\lambda_2}\right) =  \frac{1}{4\pi} \bigtriangleup\log
\left( K(z,z)+\sqrt{K(z,z)^2 - |K(z,\overline{z})|^2} \right)
\end{equation*}

\subsection{justification of \eqref{Justify} }
We must show that the following integral converges:
$$\frac{1}{2\pi} \int_X |\bigtriangleup h(z)|\cdot \mathbb{E}|\log|f(z)|\,|  dm_2(z). $$
%\mathbb{E} \left[ \frac{1}{2\pi}\int_X |\Delta h(z) \log |f(z)| | dm_2(z) \right] = \\

It is enough to prove that $\mathbb{E}|\log|f(z)|\,|$ is bounded on a compact
subset $S$ of the plane. $f(z)$ is a 2-dimensional real Gaussian variable
with parameters noted above, so we get
\begin{align*}
\mathbb{E}|\log |f(z)|| =
%\frac{1}{2\pi \sqrt{\det\Sigma}} \iint_{\mathbb{R}^2} |\log(\sqrt{x^2+y^2})|
%e^{-\frac{1}{2} (x,y)\Sigma^{-1} (x,y)^{T}} dx dy =\\
\frac{1}{2\pi \sqrt{\det\Sigma}} \iint_{\mathbb{R}^2}
|\log(\sqrt{x^2+y^2})| e^{-\frac{1}{2} (\lambda_1^{-1} x^2
+\lambda_2^{-1} y^2)} dx dy \end{align*}

%The last equality is achieved by a unitary change of variables that
%diagonalizes $\Sigma$.
As before, $\lambda_1, \lambda_2$ are the eigenvalues of $\Sigma$, dependent
on $z$. By another change of variables ($x=u\sqrt{\lambda_1},\; y
=w\sqrt{\lambda_2}$) we get:

$$\mathbb{E}|\log |f(z)|| = \frac{1}{4\pi}\iint_{\mathbb{R}^2} |\log(\lambda_1 u^2 +\lambda_2 w^2)| e^{-(u^2+w^2)/2}du\,dw$$

Fix $z$, and assume $\lambda_1\leq \lambda_2$. Let us split the integral into
two domains: $\Omega_{+} = \{(u,w)\in\mathbb{R}^2: \log(\lambda_1 u^2
+\lambda_2 w^2)\geq 0\}$ and $\Omega_{-} =\{(u,w)\in\mathbb{R}^2:
\log(\lambda_1 u^2 +\lambda_2 w^2)<0\}$.

Then, on $\Omega_{+}$ we estimate $0 < \log(\lambda_1 u^2 +\lambda_2 w^2) <
\log(\lambda_2) + \log(u^2+w^2)$. From here clearly the integral on
$\Omega_{+}$ is bounded by $C_0 +C_1 \log\lambda_2$. By lemma \ref{eigenv},
$\lambda_2 = \frac{1}{2}(K(z,z)+|K(z,\overline{z})|)$ is a continuous
function of $z$, and therefore is bounded on our compact set $S$.

For $(u,w)\in\Omega_{-}$ notice that $0> \log(\lambda_1 u^2+\lambda_2 w^2) >
\log (\lambda_1 u^2)$, therefore:
\begin{align*}
& \iint_{\Omega_{-}} |\log(\lambda_1 u^2 +\lambda_2 w^2)| e^{-(u^2+w^2)/2}du\,dw \leq \\
& \iint_{\Omega_{-}} (|\log(\lambda_1) + \log( u^2)|)
e^{-(u^2+w^2)/2}du\,dw \leq C_0 + C_1 |\log\lambda_1|
\end{align*}

Denote $m = \min\{\lambda_1(z) : z\in S\}$. If $m=0$, this leads to
$K(z_0,z_0) = 0$ for some $z_0\in K$, but this means $z_0$ is a deterministic
zero. Therefore $m>0$ and $|\log\lambda_1|$ is bounded from above.

\section{Proof of Theorem \ref{main} - Horizontal Limiting Measure}\label{sec proofs}

\subsection{Preliminaries}\label{sec prelim}
We present the probability space of our interest, equipped with a
measure-preserving transformation. We explain the notion of ergodicity in
this setup.
%We will refer to GAF's in our discussion, but all properties mentioned in this section hold for symmetric GAF's as well.

The probability space $\Omega$ is a countable product of copies of $\mathbb
C$, with $\mathbb P$ being the product of complex Gaussian measures (one on
each copy). These copies represent the random coefficients in the
construction of $f$: each $\omega=\{a_n\}_{n} \in\Omega$ corresponds to a
function $f_\omega(z) = \sum a_n \phi_n(z)$.
%(In the symmetric GAF case, $\Omega$ is a countable product of copies of $\mathbb R$, with a real Gaussian measure on each.)
$\mathcal{F}_f$ is the Borel $\sigma$-algebra generated by the basic sets
$\{\omega\in\Omega: f_\omega(z)\in B(w,r)\},$ where $z\in D, r>0$. Here
$B(w,r)=\{p\in\mathbb{C}: |p-w|<r\}$. The group of automorphisms $S_t$ shall
be defined via the correspondence $\omega \leftrightarrow f_\omega$:
$$f_{S_t \omega}(z) = f_\omega (z+t).$$
The map $S_t$ is measure-preserving, since we assumed that $f$ is stationary.
Thus, we will say the random process $f(z)$ is ergodic, if any measurable set
$A\in\mathcal{F}_f$ which is invariant to all translations ($S_t A = A,\:
\forall t\in\R$) is in fact trivial ($\mathbb P A \in \{0,1\}$).

In a similar way, one can define when is the zero-set $Z_f$ ergodic (it is
itself a random point-process in the plane). The space $\Omega$, the measure
$\mathbb{P}$ on it and the automorphisms $\{S_t\}$ are just as before. Now,
the $\sigma$-algebra $\mathcal{F}_{Z_f}$ is generated by the basic sets
$\{\omega\in\Omega: Z_{f_\omega}\cap B(z,r)\neq \emptyset\}$ with $z\in D,
r>0, B(z,r)\subset D$.

\begin{cor}
Ergodicity of $f$ implies ergodicity of $Z_f$.
\end{cor}

\begin{proof}
It is enough to prove $\mathcal{F}_{Z_f}\subset \mathcal{F}_f$. Let $A$ be a
countable dense set in $\mathbb C$. Basic sets of $\mathcal{F}_{Z_f}$ can be
written as
$$\{Z_{f_\omega}\cap B(z,r)\neq \emptyset\} =\bigcup_{m\in\mathbb N} \: \bigcap_{n\in\mathbb{N}} \: \bigcup_{p\in A\cap
B(z,r-\frac 1 m)} \left\{f_\omega(p)\in
B\left(0,\frac{1}{n}\right)\right\},$$ which is indeed in
$\mathcal{F}_f$.
\end{proof}

We will use the following classical result:
\begin{thm}
\emph{(Fomin, Grenander, Maruyama)}\label{Mar-Gren}
 A stationary GAF (symmetric or not) is ergodic w.r.t. horizontal
shifts $\{f(z) \rightarrow f(z+t)\}\}_{t\in\mathbb{R}}$ if and only if its
spectral measure $\rho$ has no atoms.
\end{thm}

This theorem was originally proved for real processes over $\mathbb{R}$ (see
for instance Grenander~\cite{Gren}), but small modifications extend it to a
strip in the complex plane for both types of functions (GAFs and symmetric
GAFs).
%The modified proof for GAFs is included in appendix \ref{app-MG}.
%It is based on the original paper by Grenander \cite{Gren}.

\subsection{Existence of the horizontal limiting measure (statement (i)\,)}\label{sec exist}
As above, for $T\ge 1$ let $\nu_{T}$ be the random locally-finite measure on
$(-\Delta, \Delta)$ defined by:
\begin{equation}\label{eq nu_T}
\nu_T(Y)=\nu_{f,T}(Y) := \frac{n_f([0, T) \times Y)}T\,, \qquad Y\subset (-\Delta,
\Delta).
\end{equation}
In this section we show that a.s. the measures $\nu_T$ converge weakly as $T$
tends to infinity.
\medskip
First, we assume that $T$ tends to infinity along positive integers. In this
case, we use the subscript $N$ instead of $T$. By a known theorem in
distribution theory (see for instance, \cite[section~2.1]{Hor}), a sequence
of measures $\nu_N$ converges weakly to some measure if and only if the
sequence of real numbers $\nu_N(h)$ is convergent for every $h\in C_0^\infty(
-\Delta, \Delta )$. It suffices to check whether $\nu_N(h)$ is convergent for
all $h\in M$, where $M\subset C_0^\infty(-\Delta, \Delta)$ is a dense set of
test-functions, and we may choose $M$ to be countable. Given test function
$h\in M$, denote by $A_h$ the event that $\nu_N(h)$ is a convergent sequence
of numbers. To prove our claim it suffices to show $\Pro(A_h)=1$ for every
$h\in M$. Note that $\nu_N(h) = \frac 1 N (X_1+X_2+\dots+X_N)$ where
\begin{equation}\label{eq:X_k}
X_k = X_k(h) = \int \done_{[k, k+1)}(x) h(y)\, dn_f(x,y)
\end{equation}
is a stationary sequence of random variables.

The random variables $X_k$ are integrable. This follows at once from an
Offord-type large deviations estimate~\cite[theorem~3.2.1]{GAF book}:
\begin{thm}[Offord-type estimate] \label{thm off}
Let $f$ be a Gaussian analytic function on a domain $D$. Then for any compact
set $K\subset D$, the number $n_f(K)$  of zeroes of $f$ on $K$ has
exponential tail: there exist positive constants $C$ and $c$ depending on the
covariance function of $f$ and on $K$ such that, for each $\lambda\ge 1$,
\[
\mathbb P \bigl\{ n_f(K) > \lambda \bigr\} < Ce^{-c\lambda}\,.
\]
\end{thm}
Therefore we can apply Birkhoff theorem \cite[chapter~7]{CrLead}. It yields
that the limit $\frac 1 N (X_1+X_2+\dots+X_N)$ almost surely exists, and so
$\Pro(A_h)=1$. This completes the proof of the weak convergence of the
sequence $\nu_N$.

\medskip
Now, we consider the general case in statement (i). Let $T\ge 1$, and let
$N=[T]$ be the integer part of $T$. Then
\[
\nu_T(h) = \frac{N}{T}\, \nu_N(h) + \underbrace{\frac1 T \int
\done_{[N, T)}(x) h(y)\, dn_f(x, y)}_{=: R_T(h)}\,.
\]
We will show that a.s. the second term on the right-hand side converges to
zero for all bounded compactly supported test functions $h$. It suffices to
prove this for all bounded test functions supported by an interval
$[-\Delta_1, \Delta_1]$ with an arbitrary $\Delta_1<\Delta$. We have
\[
|R_T(h)| \le \frac{\| h \|_\infty}{T}\, n_f \bigl( [N, N+1) \times
[-\Delta_1, \Delta_1] \bigr)\,.
\]
Employing the Offord-type estimate with $K=[0, 1]\times [-\Delta_1,
\Delta_1]$ and using translation-invariance of the zero distribution of $f$,
we see that for each $\varepsilon>0$,
\[
\mathbb P \bigl\{\, n_f ( [N, N+1]\times [-\Delta_1, \Delta_1] ) \ge
\varepsilon T \,\bigr\} = \mathbb P \bigl\{\, n_f ( [0, 1] \times
[-\Delta_1, \Delta_1] )  \ge \varepsilon T \,\bigr\} <
Ce^{-c\varepsilon N}\,.
\]
Hence, for each $M\in\mathbb N$,
\[
\mathbb P \bigl\{ \limsup_{T\to\infty} |R_T(h)| \ge \varepsilon\,
\|h\|_\infty \bigr\} \le \sum_{M}^\infty C e^{-c\varepsilon N} = C
(1-e^{-c\varepsilon})^{-1} e^{-c\varepsilon M}\, ,
\]
and we conclude that a.s.
\[
\lim_{T\to\infty} R_T(h) = 0\,
\]
for all smooth compactly supported test functions $h$. This completes the
proof of statement (i) in theorem~\ref{main}. \hfill $\Box$

\subsection{Non-random limiting measure (statement (iia),(iii)\,)}
Here we will prove that if the spectral measure $\rho_f$ has no atoms, then
the horizontal mean zero-counting-measure $\nu_f$ is not random, which is a
half of statement~(ii) (the case when $\rho_f$ is a single atom for GAFs or
two symmetric atoms for symmetric GAFs was already discussed in
remark~\ref{rmk atom}). We then compute the limit $\nu_f$, which is
statement~(iii).

Assume the spectral measure $\rho$ is continuous. From the
Fomin-Maruyama-Gre\-nan\-der Theorem we get that $f$ is ergodic, and so is
$Z_f$. Using the notation introduced in the proof of statement (i), we get
that for any smooth test function $h$, the stationary sequence of random
variables $X_k (h)$ introduced in~\eqref{eq:X_k} is ergodic. In this case the
Birkhoff ergodic theorem asserts that, a.s.,
\[
\lim_{N\to\infty} \nu_N (h) = \lim_{N\to\infty} \frac1{N}
\sum_{k=0}^{N-1} X_k (h) = \mathbb E X_0 (h)\,.
\]
Therefore, the horizontal mean zero-counting-measure is non-random, and
equals
\[
\mathbb E X_0 (h)= \mathbb E \int \done_{[0, 1)}(x) h(y)\, dn_f(x,y)
=\int \done_{[0, 1)}(x) h(y)\,\E dn_f(x,y) \,,
\]
where $\E d n_f(x,y)$ is the mean zero-counting measure. For a GAF, we
compute it directly by Edelman-Kostlan formula \eqref{eq EK}; while for a
symmetric GAF we use the formula \eqref{eq-Rdensity} (in Theorem
\ref{thm-Rdensity}). As before, denote $\psi(y) =
\displaystyle\int_{-\infty}^\infty e^{-4\pi y\lambda}d\rho(\lambda)$. Note
that
\begin{align*}
&K(z,z) = \int e^ {2\pi i \cdot 2yt} d\rho(t) = \psi(y) \\
&K(z,\overline{z}) = \int e^{2\pi i (z-\overline{z})t}d\rho(t) =
\int d\rho(t) = \psi(0),
 \end{align*}
where $z=x+iy$. Putting this into \eqref{eq-Rdensity}, we get the first
intensity of zeroes:
$$\mathbb{E} n_f = \frac{1}{4\pi}\frac {d^2}{dy^2} \log \left( \psi(y)+\sqrt{\psi(y)^2-\psi(0)^2}\right) =
\frac{1}{4\pi} \frac d{dy} \frac {\psi'(y)}
{\sqrt{\psi(y)^2-\psi(0)^2}}.$$
For any $y\neq 0$, this is a derivative in the functional sense, which equals
$S(y)$. At $y=0$, the function is not defined; but the limits
$$\displaystyle\lim_{y\rightarrow 0+}\frac {\psi'(y)}
{4\pi \sqrt{\psi(y)^2-\psi(0)^2}}= -\lim_{y\rightarrow 0-}
\frac {\psi'(y)}
{4\pi \sqrt{\psi(y)^2-\psi(0)^2}} = A$$
exist. This follows from $\frac {\psi'(y)} {\sqrt{\psi(y)^2-\psi(0)^2}}$
being an odd function, increasing in $y\in(0,\Delta)$. So, in order to compute $\E n_f$ we take
the required derivative in the distribution sense, which yields the continuous
point-wise derivative $S(y)$ (for $y\neq 0$) plus an atom of size $2A$ at 0.

In order to compute $A$ let us write this limit again, and apply
L'H\^{o}pital's rule: %(notice that both numerator and denominator tend to zero as $y$ tends to zero):
\begin{equation*}
4\pi A = \displaystyle\lim_{y\rightarrow 0+} \frac {\psi'(y)}
{\sqrt{\psi(y)^2-\psi(0)^2}} = \lim_{y\rightarrow 0+} \frac
{\psi''(y)\cdot \sqrt{\psi(y)^2-\psi(0)^2}}
{\psi(y)\cdot\psi'(y)}=(4\pi)^2 \mathcal{E}_2 \frac{1}{4\pi A}
\end{equation*}
where $\mathcal{E}_2=\displaystyle \frac {\int_{-\infty}^\infty \lambda^2
d\rho(\lambda)}{\int_{-\infty}^\infty d\rho(\lambda)}$ is the ratio between
the second and the zero moments of the spectral measure. We conclude that $A
= \sqrt{\mathcal{E}_2}$, and therefore the atom has twice this size.

\subsection{Random limiting measure (statement (iib))}
%for GAFs - see paper-v4
Our goal is to prove the second half of (ii). We present the proof for
symmetric GAFs, since it is slightly more involved. We assume that the
spectral measure has the form
$$\rho_f = c\delta_q+c\delta_{-q}+\mu,$$ where $\mu$ is a
non-trivial measure. Our goal is to show that the horizontal mean
zero-counting measure of some segment $\nu_f(a,b)$ is a non-constant random
variable. We may assume that $c=1$ and $q=1$ (if $q=0$ the analysis is
easier).

\medskip
Since $L^2_\rho(\R)$ is the direct sum of $L^2_{\delta_1+\delta_{-1}}(\R)$
and $L^2_\mu(\R)$, a union of any orthonormal bases in these subspaces is an
orthonormal basis in $L^2_\rho(\R)$. By the remark in end of section
\ref{secsub stat}, applying Fourier Transform on this union we shall get a
basis $\phi_n(z)$ from which a GAF with spectral measure $\rho$ can be
constructed. This gives the representation
$$f(z) = g(z) + \,\alpha\cos(2\pi
z)+\beta\sin(2\pi z) ,$$ where $g(z)$ is a symmetric GAF with
spectral measure $\mu$ and $\alpha,\, \beta\sim\mathcal N_\R(0,1)$ are
real Gaussians, independent of each other and of $g$. We write for short
$\eta(z)=\eta_{\alpha,\beta}(z)=\alpha\cos(2\pi z)+\beta\sin(2\pi z)$.

Fix $a,b\in(-\Delta,\Delta)$. Denote the number of zeroes of $f$ in
$[0,T]\times[a,b)$ by $N_T(g,\alpha, \beta) = \#\{z\in [0,T]\times[a,b):
g(z)=-\eta_{\alpha,\beta}(z) \}$.

Assume to the contrary that there is some constant $C$ (depending on $a$ and
$b$) such that
\begin{equation}\label{eq RlimN_T}
\text{a.s. in } \alpha,\beta, \:\: \ds \exists \lim_{T\to\infty}
\frac {N_T(g,\alpha,\beta)} {T} = C.
\end{equation}

Here we denote $\Pro_g$ and $\E_g$ the probability and expectation (resp.)
conditioned on $\alpha,\beta$. We claim that:

\begin{equation}\label{eq RGAF Elim}
\ds \E_g \lim_{T\to\infty} \frac {N_T(g,\alpha,\beta)}{T} =
\lim_{T\to\infty} \frac {\E_g N_T(g,\alpha,\beta)}{T}.
\end{equation}

This exchange is justified by the dominated convergence principle, as seen by
the following Offord-type estimate:

\begin{prop}\label{prop offT}
Let $g$ be a symmetric stationary GAF on a horizontal strip, $\alpha$ and
$\beta$ are fixed complex numbers. There exist positive constants $C$ and $c$
such that:
\begin{equation*}
\ds \sup_{T\ge 1} \Pro_g\left(\frac{N_T(g,\alpha,\beta)}{T}>s\right)
<C e^{-cs},
\end{equation*}
\end{prop}

This fact is proved in section \ref{sec offord} below.
%(For GAFs, one uses similar bounds for $N(T,\zeta) =
%\#\{z\in[0,T)\times[a,b): f(z)=-\zeta\}$, where $\zeta$ is a fixed
%complex number.)

Next we claim that the RHS of \eqref{eq RGAF Elim} is just $\E_g
N_1(g,\alpha,\beta)$. To see this, notice that for integer $T$,
$N_T(g,\alpha,\beta)$ is the sum of $T$ identically distributed random
variables, all distributed like $N_1(g,\alpha,\beta)$. This follows
immediately from stationarity of $g$ and from $1$-periodicity of
$\eta_{\alpha,\beta}(z)$. Therefore for integer $T$,
$$ \frac 1 T \,{\E_g N_T(g,\alpha,\beta)}=\E_g
N_1(g,\alpha,\beta).$$
Denote $M =\lfloor T \rfloor$. Since $\E_g N_{[M, T]}:=\E \#\{z\in
[M,T]\times[a,b): g(z)=-\eta_{\alpha,\beta}(z) \} \leq \E_g N_1<\infty$, it
follows that for non-integer $T$,

\begin{align}\label{3}
&\lim_{T\to\infty} \frac {\E_g N_T(g,\alpha,\beta)}{T} =
\lim_{T\to\infty}\left( \frac {\E_g N_{M}}{M}\cdot \frac
 {M}{T} + \frac {\E_g N_{[M, T]}}
 T\right) = \E_g N_1(g,\alpha,\beta).
\end{align}

Combining \eqref{eq RlimN_T}, \eqref{eq RGAF Elim} and \eqref{3} we have:
\begin{equation}\label{eq E(g,a,b)=C}
\text{a.s. in } \alpha,\beta, \:\: \E_g N_1(g,\alpha,\beta) = C.
\end{equation}

We divide the rest of our argument into three claims.
\begin{claim}\label{clm Rcont}
$\E_g N_1(g,\alpha,\beta)$ is continuous in $(\alpha,\beta)\in \mathbb R^2$.
\end{claim}

\begin{claim}\label{clm Rsol decr}
For any compact set $K\subset D$, let $N(g,\alpha,\beta;K)$ be the number of
solutions to $g(z)=-\eta_{\alpha,\beta}(z)$ with $z\in K$. Then $$\E_g
N(g,\alpha,\beta;K)\to n_{\cos(2\pi z)}(K)\:\:\text{ as }\:\:|\alpha|\to \infty.$$
Here $n_{\cos(2\pi z)}$ is the zero-counting measure of $\cos(2\pi z)$.
%$\mu = \ds \sum_{j\in \mathbb Z}\delta_{j+1/4}+\delta_{j+3/4}$.
\end{claim}

Relying on the last claim and \eqref{eq E(g,a,b)=C}, we get that
\begin{equation}\label{eq-sinlike measure}
\E_g N_1(g,\alpha,\beta) = 2\delta_0([a,b))
\end{equation}
for almost all $\alpha,\beta$. Since $\E_g N_1(g,\alpha,\beta)$ is continuous
in $\alpha,\beta$, \eqref{eq-sinlike measure} is true \emph{for all}
$\alpha,\beta$, and in particular for $(\alpha,\beta)=(0,0)$. The following
claim asserts this happens only for one family of symmetric GAFs:

\begin{claim}\label{clm Rno zeroes}
If for $-\Delta<a<0<b<\Delta$,
$$\E_g N_1(g,0,0) = \E n_g([0,1)\times[a,b) = 2\delta_0([a,b)),$$
then the spectral measure of $g$ is $\frac 1 2 (\delta_1+\delta_{-1})$ (up to
a constant multiplier).
\end{claim}

From this last claim it follows that the spectral measure of $f$ consists
only of symmetric atoms at $\pm1$, which contradicts our assumption.

\medskip
It remains now to prove the claims. In the course of their proof we justify
exchange of limits and expectations by the following

\begin{prop}[Offord-type estimate for sine-like levels] \label{prop off2}
Let $g$ be a symmetric Gaussian analytic function on a domain $D$, and let
$\alpha$ and $\beta$ be fixed complex numbers. Then for any compact $K\subset
D$, the number $N(g, \alpha,\beta ;K)$ of solutions to
$g(z)=-\eta_{\alpha,\beta}(z)$ with $z\in K$ has exponential tail: There
exist positive constants $C$ and $c$ such that
\begin{equation*}
\Pro(N(g, \alpha,\beta ;K)>s)\leq Ce^{-cs}.
\end{equation*}
%where $C$ and $c$ are positive constants depending on $\alpha$,
%$\beta$, on the covariance function of $g$, and on $K$.
\end{prop}

This proposition's proof is much similar to that of proposition \ref{prop
offT}, and we omit it.

\begin{proof} [Proof of claim \ref{clm Rcont}]
Fix $-\Delta< \alpha_0<\beta_0 <\Delta$. It is clear that almost surely in
$g$, $N_1(g,\alpha,\beta_0)$ approaches $N_1(g,\alpha_0,\beta_0)$ as $\alpha$
approaches $\alpha_0$ (the event of having a solution to
$g(z)=-\eta_{\alpha_0,\beta_0}$ on the boundary of $[0,1]\times[a,b]$ is
negligible).

By proposition ~\ref{prop off2}, we may pass to the limit:
\begin{align*}
&\lim_{\alpha\to\alpha_0}\E_g N_1(g,\alpha,\beta_0) = \lim_{\alpha\to\alpha_0}\int \Pro_g(N_1(g,\alpha,\beta_0)>s) ds =\\
& = \int \lim_{\alpha\to\alpha_0}\Pro_g(N_1(g,\alpha, \beta_0)>s) ds
= \int \Pro_g(N_1(g,\alpha_0,\beta_0)>s) ds = \E_g(N_1(g,\zeta_0)).
\end{align*}

\end{proof}

\begin{proof} [Proof of claim \ref{clm Rsol decr}.]
Fix $\beta$ and $g$. For any $\alpha\neq0$, the zeroes of
$$h_\alpha(z) = \frac{g(z) + \sin(2\pi z)}{\alpha}+\cos(2\pi z)$$
and of $f(z) = g(z)+\eta_{\alpha,\beta}(z)$ are identical. Now notice that
$h_\alpha(z)$ converges locally uniformly to $\cos(2\pi z)$ as $\alpha\to
\infty$ (i.e., uniformly on any compact set). By Hurwitz's Theorem, this
implies that the zero counting measures also converge locally uniformly, in
the sense that for any compact $K\subset D$,
$$\lim_{\alpha\to\infty} n_{h_\alpha}(K) = n_{\cos(2\pi z)}(K).$$ By the bound in proposition \ref{prop off2}, this almost
sure convergence in $g$ yields moment convergence:
$$\E_g n_{h_\alpha}(K) \to n_{\cos(2\pi z)}(K), \,\, \text{as } \alpha\to\infty.  $$

\end{proof}

\begin{proof} [Proof of claim \ref{clm Rno zeroes}.]
Suppose the spectral measure is normalized, so that $\psi(0) = \int_\R
d\rho(\lambda) = 1$ (else, multiply it by a constant). The premise and
theorem \ref{main} give two conditions on $\psi(y)=K(iy,iy)$:
\begin{align*}
& \frac {\psi'(y)}{\sqrt{\psi(y)^2-1}} = c\, , & R =
2\sqrt{\int_{\mathbb R}\lambda^2 d\rho(\lambda)} = 2\, ,
\end{align*}
for some constant $c\in \mathbb R$. Solving the left-side ordinary
differential equation, and using $\psi(0)=1$, we get $\psi(y) = \cosh(cy)$.
Since $\psi$ is a Laplace transform of $\rho$, we get $\rho = \frac 1 2
(\delta_{c/2\pi} + \delta_{-c/2\pi} )$. But the right-side equation would be
satisfied only if $c=2\pi$.
\end{proof}

\appendix

\section{Convergence on all intervals (proof of proposition \ref{prop all segs})}\label{app all segs}
In this section we prove proposition \ref{prop all segs}. We use the
notations developed in section \ref{sec prelim}. For any point in the
probability space $\omega\in\Omega$, let $\nu^\omega_N$ be the sequence of
measures introduced in \eqref{eq nu_T}, for integer $T=N$ (the non-integer
case follows just as in the proof of part (i) of Theorem \ref{main}, and will
not be discussed).

Define the set
$$C = \{\omega\in\Omega: \, (\nu^\omega_N)_N \text{ converges weakly} \} $$
Notice that by part (i) of Theorem \ref{main}, $\Pro(C)=1$. For convenience
of notation we denote the limit $\nu_f=\nu^\omega$. From general measure
theory, one can deduce that almost surely, $\nu_N([a,b))$ converges to
$\nu^\omega([a,b))$ for all $a,b$ out of a countable exceptional set. This
exceptional set is the set of atoms of $\nu_f$ (which might be random). We
thus turn to define

$$A = \{\omega\in\Omega: \, \lim_{N\to\infty} \nu^\omega_N\{a\}=\nu^\omega\{a\} \text{, for each atom } a \text{ of }\nu^\omega\}\subset C. $$

\begin{claim}\label{clm A meas}
$A$ is measurable with respect to $\mathcal F_f$.
\end{claim}
The proof of this claim will be presented in the end of this section. Our
next goal would be to prove:
\begin{claim}\label{clm ProA=1}
$\Pro(A) = 1$.
\end{claim}

Our main tool will be the Ergodic Decomposition Theorem (proved, for
instance, in \cite[chapter 2.2.8]{aaro}):
\begin{thm}[Ergodic Decomposition]
Let $(\Omega,\mathcal F,\Pro)$ be a standard Borel-space, equipped with a
measure preserving transformation $S:\Omega\to\Omega$. Then the set
$E^S(\Omega)$ of ergodic probability measures on $\Omega$ is not empty, and
there exists a map $\beta:\Omega\to E^S(\Omega)$ such that for any measurable
set $A\in\mathcal F$ the following holds:
\begin{enumerate}
\item the map $\Big\{\begin{array}{c}
\Omega\to[0,1] \\
\omega\mapsto \beta_\omega(A)
\end{array}
$ is measurable.
\smallskip
\item $\Pro(A) = \int_\Omega \beta_\omega(A) d\Pro(\omega). $
\end{enumerate}
\end{thm}

\medskip
\begin{proof}[Proof of claim \ref{clm ProA=1}]
The stationary system $(\Omega,\mathcal F_{Z_f},\Pro,S)$ defined in section
\ref{sec prelim} and the set $A$ defined above meet the requirements of the
Ergodic Decomposition Theorem.
%(Notice we deal here with a discrete group of transformations $\{S^n\}_{n\in\N}$).
Therefore, in order to prove our claim it is enough to show that
$$\forall \gamma\in E^S(\Omega), \gamma(A) = 1. $$

Fix an $S$-ergodic measure $\gamma$. Since $A$ is an $S$-invariant set, we
get $\gamma(A)\in\{0,1\}$. Moreover, the event $\{\nu^\omega \text{ has an
atom in the interval } I\}$ is also invariant, for any interval
$I\subset(-\Delta,\Delta)$. Therefore, $\gamma$-a.s. the limiting measure
$\nu^\omega$ has atoms at some known points
$(a_n)_{n\in\N}\subset(-\Delta,\Delta)$.

For a certain atom $a=a_n$, define the stationary sequence:
$$X_k(a) = n_f([k,k+1)\times\{a\} ),$$
and notice that
$$\nu_N\{a\} = \frac 1 N \sum_{k=0}^{N-1} X_k(a).$$
As $\gamma$ is ergodic, we have by Birkohff's ergodic Theorem:
$$\gamma\text{-a.s. } \nu^\omega_N\{a\} \text{ converges to } \E_\gamma n_f([0,1)\times \{a\}) = \nu^\omega\{a\}, \text { as } N\to\infty$$
Since there are at most countably many atoms, we get $\gamma(A)=1$.
\end{proof}
\medskip
We now know that $\Pro$-a.s., the sequence $\nu_N$ is weak convergent and
converges on any atom of the limiting measure (to the desired limit). A
general claim from measure theory will assure us that in this case, $\nu_N$
converge on any interval:

\begin{claim}
Suppose $(\nu_N)_N$ is a weak-converging sequence of measures on some
interval $I$, and let $\nu$ be the limiting measure. If
$\ds\lim_{N\to\infty}\nu_N\{a\}= \nu\{a\}$ for every atom $a$ of $\nu$, then
$\ds\lim_{N\to\infty}\nu_N(J)= \nu(J)$ for \emph{every} interval $J\subset
I$.
\end{claim}

\begin{proof}
\begin{comment}
Suppose $J=[a,b)$. If $a$ and $b$ are not atoms of $\nu$, then taking
$\phi_k\in C_0$ such that $\sup_{I}{\phi_k} = 1$, $\phi_k$ equals $1$ on
$(a+\frac 1 k,b-\frac 1 k)$ and $0$ outside $[a,b]$. Then $\phi_k$ converge
to $\phi=\done_{[a,b]}$ in the sense that for every bounded locally finite
measure $\mu$, $\mu(\phi_k)\to\mu(\phi)$, as $k\to \infty$. We claim that
$$\ds \lim_{N\to\infty} \lim_{k\to\infty} \nu_N(\phi_k)=\lim_{k\to\infty}\lim_{N\to\infty}\nu_N(\phi_k)$$
which is justified by the Moore-Osgood theorem (for fixed $k$, the
convergence in $N$ is uniform, while for fixed $N$ we have point-wise
convergence in $k$). Therefore, the desired limit exists:
$$\ds \lim_{N\to\infty}\nu_N(\phi) = \lim_{k\to\infty}\nu(\phi_k) = \nu(\phi).$$
\end{comment}

 We demonstrate the case $J=[a,b)$, where $\nu$ has no atom at $b$ (other cases are similar).

Given $\ep>0$, one can construct piecewise linear functions $\phi^+,\,\phi^-
\in C(I)$ such that:
\begin{equation}\label{eq inds}
\forall x,\, \phi^-(x)\leq \done_{(a,b)} \leq \done_{[a,b)}(x) \leq
\phi^+(x),
\end{equation}
and additionally
$$0<\nu(\phi^+)-\left(\nu(\phi^-)+\nu\{a\} \right)<\epsilon.$$
%(Those are upper and lower continuous estimates of $\done_{[a,b)}$).

%$\text{supp}\phi^+ =[a-\delta, b]$, $\phi^+|_{}

(For instance, for large enough parameter $n$, take $\phi^+$ supported on
$[a-\frac 1 n, b]$, equals $1$ on $[a,b-\frac 1 n]$ and linear otherwise;
$\phi^-$ supported on $[a,b]$, equals $1$ on $[a+\frac 1 n,b-\frac 1 n]$, and
linear otherwise).

By applying the measure $\nu_N$ to relation \eqref{eq inds}, we get:
$$\nu_N(\phi^-)+ \nu_N\{a\}\leq \nu_N([a,b)) \leq \nu_N(\phi^+)$$

But, from our assumptions, for large enough $N$ we have
$$ \nu(\phi^-)+\nu\{a\}-\epsilon \leq \nu_N([a,b)) \leq \nu(\phi^+)+ \epsilon $$
As the difference between those bounds does not exceed $3\epsilon$, we see
the limit $\lim_{N\to\infty} \nu_N([a,b))$ exists. Since $\nu(\phi^+)$ is as
close as we want to $\nu([a,b))$, we are done.

\end{proof}

It remains only to prove the measurability of $A$.
\begin{proof}[Proof of claim \ref{clm A meas}]
We first investigate some underlying objects. Denote by $P=P(-\Delta,\Delta)$
the space of all locally finite measures on $(-\Delta,\Delta)$ induced with
the L\'evy - Prokhorov metric (for which convergence in metric is equivalent
to weak convergence):
$$\pi(\mu,\nu) := \inf\{ \ep >0\, |\,\, \forall Y\in\mathcal B \,\mu(Y)\leq \nu(Y^\ep) + \ep \text{ and } \nu(Y)\leq\mu(Y^\ep)
+\ep\}, $$ where $\mathcal B$ is the sigma-algebra of Borel subsets
of $(-\Delta,\Delta)$, and $Y^\ep = \ds \cup_{p\in Y} B(p,\ep)$ is an
$\ep$-neighborhood of $Y$.

We claim that the map
$$\omega \mapsto \nu^\omega_1(\cdot) = n_{f_\omega}([0,1)\times \cdot) \in P$$
is measurable; in fact, it is continuous (A small change of the coefficients
$\omega=(a_1,a_2,\dots)$ in $l^2$ sense will yield a small change in the
counting measure of zeroes $\nu^\omega_1$ in L\'evy - Prokhorov sense).

Now consider the space $X = P^\N$ of sequences of measures with the product
topology. Notice that the map $\Omega\to X$ defined by $\omega \mapsto
\{\nu_N^\omega\}$ is measurable, as each coordinate is measurable; Moreover,
its image lies almost surely in the (measurable subset) of weak converging
sequences. The map $C \to P$ which takes a weak converging sequence
$(\nu_N)\in C$ to its limit $\nu\in P$ is also measurable. We arrive at
\begin{obs}\label{obs-meas}
Any measurable set $M\in P$ induces a measurable set $\widetilde M =
\{\omega:\,\nu^\omega\in M \}\subset C \subset\Omega$.
\end{obs}
Consider the event:
$$B=\{\omega\in\Omega: \, \text{ the limiting measure } \nu^\omega \text{ has at least one atom} \}\subset C
\subset X$$

By the last observation, $B$ is measurable w.r.t. $\mathcal F$.
\begin{comment}
since:
$$B = \cup_{c\in\Q^+} \cap_{\ep\in\Q^+} \cup_{|I|<\ep} \cup_{N} \cap_{n>N} \{\nu_n(I)>c\}$$
where $I$ is an interval with rational endpoints.
\end{comment}

We construct a measurable function $h:\, B\to (-\Delta,\Delta)^{\N}$ which
maps $\omega\in B$ to a list of all atoms of the limiting measure
$\nu^\omega$, as follows. Let $h_1: \, B\to (-\Delta,\Delta)$ be the map
which maps some $\omega\in B$ to the largest atom among those of $\nu^\omega$
(if some (finite) number of atoms share this property, return the left-most
one). Again by observation \ref{obs-meas}, $h_1$ is a measurable map.
\begin{comment} the event "$\nu$'s largest atom is less then $a$" can be expressed as:
\begin{align*}
&\exists c_0\, &\forall n \,\exists p_n,q_n\in\Q, \, p_n<q_n<a \,:
\,
\nu(p_n,q_n)>c_0, \, q_n -p_n\to 0 \\
&& \forall \ep\, \forall r_n,s_n\in \Q,\, a<s_n<r_n:\,
\nu(r_n,s_n)<C_0+\ep.
\end{align*}
\end{comment}
In a similar manner we construct $h_2$, which gives the second (left-most)
largest atom; and so forth. Our list of atoms is simply $h=(h_1,h_2,\dots)$.
We notice that
$$A = \ds \cap_{i\in\N} \{\omega: \,
(\nu^\omega_N\{h_i\omega\})_N \text{ is a convergent sequence of
numbers} \}=:\cap_{i\in\N}E_i.$$ All that remains is to prove
measurability of $E_1$.

Indeed, the map $H: X\times (-\Delta,\Delta) \to \{0,1\}$ which matches
$(\{\nu_N\},a)$ to the indicator of the event $\{ (\nu_N\{a\})_N \text{ is a
convergent sequence}\}$ is measurable; by composition of measurable maps
$\done_{E_1} = H((\nu_N^\omega),h_1\omega)$ is a measurable function, as
anticipated.
\end{proof}

\section{Exponential Decay of Some Tail Events - Offord-Type Estimates}\label{sec offord}

In the course of the proof the main theorem, we used several times
exponential estimates on certain probabilities: theorem \ref{thm off},
propositions \ref{prop off2} and \ref{prop offT}; and similar propositions
for GAFs. Such estimates are sometimes referred to as "Offord-type large
deviations estimates". We demonstrate the proof of proposition \ref{prop
offT}, as the rest are similar. We adopt the proof of Sodin \cite{Sod-Off},
presented also in \cite[chapter~7]{GAF book}.

We first present our key-lemma, which deals with $2$-dimensional Gaussian
random variables.
\begin{lem}\label{lem log bound}
If $\eta \sim \mathcal N_{\R^2} (\mu,\Sigma)$, and $E$ is an event in the
probability space with $\Pro(E)=p$, then:
$$ |\E (\chi_E\log|\eta|)|  \leq p \left[-(1+\frac 1 {2\lambda_1}) \log p+ \frac p {4\lambda_1} +
\frac 1 2 \log( \text{trace }\Sigma + |\mu|^2)\right],$$ where
$\lambda_1$ is the biggest eigenvalue of $\Sigma$.
\end{lem}

\begin{proof}
\emph{Upper bound:} by Jensen's inequality,
$$
\frac 1 p \E(\chi_E \log |\eta|^2)\leq \log
\left(\frac{\E(|\eta|^2\chi_E)}{p} \right) \leq \log \E|\eta|^2 -
\log p.
$$
If $\eta=u+iv$, then
$$\E |\eta|^2 = \E u^2+E v^2 = \text{var }u + (\E u)^2 + \text{var v} + (\E v)^2 = \text{trace } \Sigma + |\mu|^2 $$
Putting this in the previous equation, we get:
\begin{equation}\label{eq-upper}
\E (\chi_E\log|\eta|)  \leq  \frac p 2 [\log(\text{trace
}\Sigma+|\mu|^2)- \log p].
\end{equation}

\emph{Lower bound: }
\begin{align*}%\label{eq-lower}
\E (\chi_E\log|\eta|) &\geq - \E (\chi_E\log^-|\eta|) \\
& = -\E(\log^-|\eta| \chi_{E\cap \{|\eta|<p\}})-\E(\log^-|\eta|
\chi_{E\cap \{|\eta|>p\}})
\end{align*}
The second term may be bounded below by
\begin{equation}\label{eq-second}
-\E(\log^-|\eta| \chi_{E\cap \{|\eta|>p\}})\geq p\log p
\end{equation}

For the first term, we begin with some general manipulations:
\begin{align*}
-\E(\log^-|\eta| \chi_{E\cap \{|\eta|<p\}}) &\geq  -\E(\log^-|\eta|
\chi_{\{|\eta|<p\}})
=-\E \left[\chi_{|\eta|\leq p} \int_0^1 \chi_{s>|\eta|}\frac {ds}{s} \right] %\label{eq-first manipul}
\\
&=-\int_0^1 \Pro[|\eta|< \min(p,s)]\frac {ds}{s}\notag
%&= -\log (s) \: \Pro[|\eta|<\min(p,s)]|_0^1 + \int_0^1\log (s)\: \frac d {ds} \Pro[|\eta|<\min(p,s) ]ds
\end{align*}

Let us therefore bound from above the probability $\Pro (|\eta| < R)$. Denote
by $\lambda_1,\,\lambda_2$ the eigenvalues of $\Sigma$, where
$\lambda_1\geq\lambda_2\ge 0$.

\begin{align*}
\Pro (|\eta| < R) & = \ds \frac 1 {2\pi \sqrt{|\Sigma|}}\int_{|x|<R}  \exp\left(-\frac 1 2 (x-\mu)^{T}\Sigma ^{-1}(x-\mu) \right) dm_2(x) \\
& \leq \frac 1 {2\pi \sqrt{|\Sigma|}} \int_{|x|<R}\exp\left(-\frac 1 2 x^{T}\Sigma ^{-1}x \right)  dm_2(x) \\
& \leq \frac 1 {2\pi} \int_{|y|<R} \exp\left(-\frac 1 2 (\lambda_1^{-1}y_1 + \lambda_2^{-1}y_2)\right)dm_2(y)\\
%& = \frac 1 {2\pi } \int_{\{\lambda_1 w_1^2 + \lambda_2 w_2^2 < R^2\}} \exp\left(-\frac 1 2 |w|^2\right)dm_2(w) \\
%&\leq \frac 1 {2\pi} \int_{\{\lambda_1 w_1^2 + \lambda_1 w_2^2<R^2\}} \exp\left(-\frac 1 2 |w|^2\right)dm_2(w) \\
%=
&\leq \int_0^{R/\sqrt{\lambda_1}}e^{-\frac 1 2 r^2}r dr = 1-
e^{-\frac {R^2}{2\lambda_1}}
\end{align*}
We have used the changes of variables $y = U x$ where $U$ is the orthogonal
matrix diagonalizing $\Sigma$, and later $y_i = \sqrt{\lambda_i}w_i$ for
$i=1,2$.

Continuing, we have:
\begin{align*}
-\E&(\log^-|\eta| \chi_{E\cap \{|\eta|<p\}}) \geq  -\int_0^p \frac
{1-e^{-s^2/2\lambda_1}}{s} ds -
\int_p^1 \frac {1-e^{-p^2/2\lambda_1}} s ds \notag\\
%& = -(1-e^{-s^2/2\lambda_1})\log\, s\, \Big|_0^p +
%\int_0^p \frac {s}{\lambda_1} e^{-s^2/2\lambda_1}\log\,s \,ds - (1-e^{-p^2/2\lambda_1})\log\,s\, \Big|_p^1 \notag \\
& = \frac 1 2 \int_0^{p^2/2\lambda_1} e^{-t}\log(2\lambda_1t)dt \notag \\
& \ge \frac 1 2 \int_0^{p^2/2\lambda_1} \log(t) dt +\frac
{p^2}{4\lambda_1} \log (2\lambda_1)
%= \frac {p^2}{4\lambda_1}\left(\log \frac p {2\lambda_1} -1 + \log (2\lambda_1)\right) \\
= \frac {p^2}{2\lambda_1}(\log p -\frac 1 2 )
\end{align*}

Therefore our lower bound is
 \begin{align*}
 \E (\chi_E\log|\eta|) &\ge \frac {p^2}{2\lambda_1}(\log p -\frac 1 2) - p\log p
\geq -\frac {p^2}{4\lambda_1} + (1- \frac 1 {2\lambda_1}) p \log p
\end{align*}

Combining the two bounds we get the desired result.
\end{proof}

We now turn to the proof of proposition \ref{prop offT}. Take $\phi(z) =
\phi_T(z)$ a real $C^2$ function, whose support is $[-\frac 1 2,T+\frac 1
2]\times[a^\prime,b^\prime]$ with $-\Delta<a^\prime<a<b<b^\prime<\Delta$, and
which takes the value $1$ on $[0,T]\times[a,b)$. We may build such
$\phi_T(z)$ that will obey also the bound $\norm{\Delta\phi}_{L^1} <
10(T+b-a)$. Assume $\alpha$ and $\beta$ are fixed for now, and fix also
$s>0$. We are interested in dominating the probability of the event $A_T
=\{N_T > sT \} $. Write $p = p_T = \Pro(A_T)$.

We have
\[
N_T < \frac 1 {2\pi}\int \Delta\phi_T(z) \log|f(z)| dm_2(z)\,,
\]

%\[
%N_T < \frac 1 {2\pi}\int \Delta\phi_T(z)
%\log|g(z)+\eta_{\alpha,\beta}(z)| dz\,,
%\]
and therefore,
\begin{align*}
sT \cdot p &\leq \E_g (\chi_{A_T} N_T) \leq \E_g \left( \chi_{A_T}
\frac 1 {2\pi}\int \Delta\phi(z) \log|f(z)| dm_2(z)  \right) \\
&= \frac 1 {2\pi} \int \Delta\phi \: \E_g \left(\chi_{A_T} \log|f(z)| \right) dm_2(z) \\
&\leq \frac {1}{2\pi} \norm{\Delta \phi}_{L^1} \sup_{z\in D} \E_g
\left(\chi_{A_T} \log|f(z)| \right)
\end{align*}

Before we continue, let us justify the exchange of expectation and integral.
Recall $f(z) = g(z) + \eta_{\alpha,\beta}(z)$; so in order to use Fubini's
theorem we need
\begin{equation}\label{eq fubini log}
\int_D \E_g |\Delta\phi(z) \cdot \log |g(z)+\eta_{\alpha,\beta}(z)|
| dm_2(z)  = \int_D |\Delta\phi(z) |\, \E_g|\log
|g(z)+\eta_{\alpha,\beta}(z)|\, | <\infty
\end{equation}

For each $z\in D$, $f(z) = g(z)+\eta_{\alpha,\beta}(z)$ is a $2$ dimensional
Gaussian random variable, with mean $\mu(z)=\eta_{\alpha,\beta}(z)$, and the
same covariance matrix $\Sigma(z)$ as the $2$ dimensional Gaussian r.v.
$g(z)$. By lemma \ref{eigenv}, we see that both $\mu(z)$ and $\Sigma(z)$
depend continuously on the paremeter $z$. So, the function $\E_g|\log
|g(z)+\eta_{\alpha,\beta}(z)|$ is bounded above for $z\in
\text{support}(\phi)$, which ends this argument.

Notice further, that in our stationary case $\lambda_1(z)$, $\lambda_2(z)$,
the eigenvalues of $\Sigma(z)$, depend on $y$ only, where $z=x+iy$. Therefore
they have lower and upper bounds on $\R \times [a^\prime,b^\prime]$. Notice
that also $\mu(z)$, being a trigonometric function, has such bounds. By
applying lemma \ref{lem log bound} with $\eta(z) =
g(z)+\eta_{\alpha,\beta}(z)$, we get:

$$ \sup_{z\in \R\times[a^\prime,b^\prime]} \E (\chi_{A_T} \log|g(z)+\zeta|) < p(c_1 -c_2 \log p).$$
where $c_1,c_2$ are positive constants ($c_1$ depending on $\alpha,\beta$,
the horizontal lines $a,b$, and the kernel of $g$). Putting all this
together, we get:
$$sT\cdot p \leq \frac 5 \pi (T+b-a) p (c_1 - c_2 \log p ) ,$$
which leads to the exponential bound we strived for:
$$\exists c,C>0 \:\text{such that } p_T = \Pro_g(N_T > T\,s) \leq C e^{-c s }, \: \forall T\geq 1\,. $$
\hfill $\Box$

\end{document}